\documentclass[11pt]{article}
\usepackage{amssymb,float,amsthm,hyperref,xcolor,amsmath,tipa}
\hypersetup{
    colorlinks,
    citecolor=violet,
    filecolor=blue,
    linkcolor=blue,
    urlcolor=violet
}
\usepackage{thmtools}

\newtheorem{theorem}{Theorem}[section]
\newtheorem{lemma}[theorem]{Lemma}

\theoremstyle{definition}
\newtheorem{remark}[theorem]{Remark}

\newtheorem{problem}[theorem]{Problem}
\newtheorem{conjecture}[theorem]{Conjecture}
\newtheorem*{thm1.1PartI}{Theorem 1.1(i)}
\newtheorem*{thm1.1PartII}{Theorem 1.1(ii)}
\newtheorem*{thm1.4PartI}{Theorem 1.4(i)}
\newtheorem*{thm1.4PartII}{Theorem 1.4(ii)}

\usepackage[margin=1in]{geometry}
\counterwithin{figure}{section}

\title{A sharp lower bound for some reciprocal Rado numbers}
\author{Collier Gaiser\thanks{Department of Mathematics, Community College of Aurora, Aurora, CO 80011,  United States of America. Email: {\tt colliergaiser@gmail.com}}
\and 
Mojtaba Ramezanpour\thanks{Community College of Aurora, Aurora, CO 80011, United States of America. Email: {\tt mrramezanpour7698@gmail.com}}
}
\date{}
\begin{document}
\maketitle

\maketitle
\begin{abstract}
Let $f_r(k)$ be the smallest $n$ such that every $r$-coloring of $\{1,2,\ldots,n\}$ has a monochromatic solution to the equation
\[
\frac{1}{x_1}+\frac{1}{x_2}+\cdots+\frac{1}{x_k}=\frac{1}{x_{k+1}},
\]
where $x_1,x_2,\ldots,x_k$ are not necessarily distinct. In this paper, we prove that $f_r(2)\geq 4^r/2$ for all $r\geq1$, and $f_r(k)\geq(2^r-1)k^r$ for all $k\geq3$ and $r\geq1$. When $r=2$, we show that, if $k=3\cdot2^m$ for some positive integer $m$, then $f_2(k)=3k^2$; and if $k=p^m$ for some odd prime number $p$ and positive integer $m$, then $f_2(k)\geq3k^2+1$. We also provide new computational results for $f_2(k)$ and $f_3(k)$, as well as a generalization of our lower bounds for $f_2(k)$ to equations with general coefficients.

\end{abstract}

{\small \textbf{Keywords:} arithmetic Ramsey theory, unit fractions, Rado numbers} \\
\indent {\small \textbf{AMS 2020 subject classification:} 05D10; 11B75}

\maketitle

\section{Introduction}
In arithmetic Ramsey theory, by the celebrated Rado's theorem \cite{Rado1933} (see also \cite[Chapter 9]{LandmanRobertson2014}), for all integers $k\geq2$ and $r\geq1$, if $n$ is large enough, then every $r$-coloring of $\{1,2,\ldots,n\}$ has a monochromatic solution to the equation
\begin{equation}\label{Equation:Linear}
x_1+x_2+\cdots+x_k=x_{k+1}.
\end{equation}
Let $R_r(k)$ be the smallest positive integer $n$ such that every $r$-coloring of $\{1,2,\ldots,n\}$ has a monochromatic solution to Equation~(\ref{Equation:Linear}), where $x_1,x_2,\ldots,x_k$ are not necessarily distinct. Due to Schur's theorem \cite{Schur1916} (see also \cite[Chapter 8]{LandmanRobertson2014}), the values of $R_r(2)$ are known as Schur numbers and it is known that $R_1(2)=2$, $R_2(2)=5$, $R_3(2)=14$, and $R_4(2)=45$. Only very recently in 2018, using a computer-generated proof that required two petabytes of space, Heule \cite{Heule2018} determined that $R_5(2)=161$. All other Schur numbers are still undetermined. We refer interested readers to \cite{ACPPRT2022} for recent developments on Schur numbers.

Schur \cite{Schur1916} (see also \cite[Theorem 8.9]{LandmanRobertson2014}) proved that $R_r(2)\geq(3^r-1)/2+1$ for all $r\geq1$. Zn\'{a}m \cite{Znam1966} generalized this result and proved that, for all $k\geq2$ and $r\geq1$,
\[
R_r(k)\geq\frac{k-1}{k}\left[(k+1)^r-1\right]+1.
\]
We note that Zn\'{a}m \cite{Znam1966} proved the above lower bound in 1966, and then Beutelspacher and Brestovansky~\cite{BeutelspacherBrestovansky1982} reproved this lower bound in 1982. When $r=2$ and $3$, the exact values for $R_r(k)$ are known for all $k$ and they coincide with the lower bound of Zn\'{a}m: Beutelspacher and Brestovansky~\cite{BeutelspacherBrestovansky1982} proved that $R_2(k)=k^2+k-1$ for all $k\geq2$, and, recently in 2019, Boza, Mar\'{i}n, Revuelta, and Sanz \cite{BMRS2019} confirmed that $R_3(k)=k^3+2k^2-2$ for all $k\geq2$.

Brown and R\"{o}dl \cite{BrownRodl1991}, and Lefmann \cite{Lefmann1991} independently proved that, among other things, Rado's theorem still holds when the variables are replaced with their reciprocals. This implies that, for all positive integers $k\geq2$ and $r\geq1$, if $n$ is large enough, then every $r$-coloring of $\{1,2,\ldots,n\}$ has a monochromatic solution to the equation
\begin{equation}\label{Equation:Main}
\frac{1}{x_1}+\frac{1}{x_2}+\cdots+\frac{1}{x_k}=\frac{1}{x_{k+1}}.
\end{equation}
Let $f_r(k)$ be the smallest positive integer $n$ such that every $r$-coloring of $\{1,2,\ldots,n\}$ has a monochromatic solution to Equation~(\ref{Equation:Main}), where $x_1,x_2,\ldots,x_k$ are not necessarily distinct. Brown and R\"{o}dl \cite{BrownRodl1991} proved that $f_2(k)\leq k^2(k^2-k+1)(k^2+k-1)$ and the first author \cite{Gaiser2024} improved this upper bound to $f_2(k)\leq6k(k+1)(k+2)$ in 2024. Using a result of Boza, Mar\'{i}n, Revuelta, and Sanz \cite{BMRS2019}, the first author \cite{Gaiser2024} also showed a polynomial upper bound for $f_3(k)$.

As for the lower bound, Tejaswi and Thangdurai \cite{TejaswiThangadurai2002} proved that $f_r(2)\geq3\cdot2^{r+3}$ for all $r\geq3$, and the first author \cite{Gaiser2024} observed that $f_r(k)\geq k^r$ for all $k\geq2$ and $r\geq1$. In this paper, we first improve both lower bounds.
\begin{theorem}\label{Theorem:Main}
For all $r\geq1$, we have
\[
f_r(2)\geq 4^r/2;
\]
and, for all $k\geq3$ and $r\geq1$, we have \[f_r(k)\geq(2^r-1)k^r.\]
\end{theorem}

When $r=2$, the lower bound $f_2(k)\geq3k^2$ for $k\geq3$ in Theorem~\ref{Theorem:Main} is sharp. The equality is reached for infinitely many values of $k$; however, there are also infinitely many values of $k$ such that $f_2(k)>3k^2$. We say that $k$ is an odd prime power if $k=p^m$ for some odd prime number $p$ and positive integer $m$. We prove the following:
\begin{theorem}\label{Theorem:2-ColoringsCombined}
If $k=3\cdot2^m$ for some positive integer $m$, then $f_2(k)=3k^2$; and if $k$ is an odd prime power, then $f_2(k)\geq3k^2+1$.
\end{theorem}

Myers and Parrish~\cite{MyersParrish2018} computed $f_2(k)$ for $k\in\{2,3,4\}$ and we provide more computational results by calculating $f_2(k)$ for $k\in\{5,6,\ldots,25\}$. The key feature of these computational results is that, for all $k\in\{3,4,\ldots,25\}$, we have $f_2(k)=3k^2$ if and only if $k$ is an odd prime power. Hence, we make the following conjecture:
\begin{conjecture}
    If $k\geq4$ is not an odd prime power, then $f_2(k)=3k^2$.
\end{conjecture}

When $x_1,x_2,\ldots,x_k$ have arbitrary coefficients, it is not guaranteed that there exists $n$ such that every $r$-coloring of $\{1,2,\ldots,n\}$ has a monochromatic solution to the equation
\begin{equation}\label{Equation:GeneralLinear}
    a_1x_1+a_2x_2+\cdots+a_kx_k=x_{k+1},
\end{equation}
where $a_1,a_2,\ldots,a_k$ are positive integers and $k\geq2$. However, it is still true that, for any positive integer $k\geq2$, if $n$ is large enough, then every $2$-coloring of $\{1,2,\ldots,n\}$ has a monochromatic solution to Equation~(\ref{Equation:GeneralLinear}) (see \cite[Theorem~9.2]{LandmanRobertson2014} or \cite[Theorem~0.1]{MR2007}). Let $R(a_1,a_2,\ldots,a_k)$ be the smallest positive integer $n$ such that every $2$-coloring of $\{1,2,\ldots,n\}$ has a monochromatic solution to Equation~(\ref{Equation:GeneralLinear}), where $x_1,x_2,\ldots,x_k$ are not necessarily distinct. Hopkins and Schaal \cite{HS2005}, and Guo and Sun~\cite{GS2008} proved that \[R(a_1,a_2,\ldots,a_k)=\min\{a_1,a_2,\ldots,a_k\}(a_1+a_2+\cdots+a_k)^2+a_1+a_2+\cdots+a_k-\min\{a_1,a_2,\ldots,a_k\}.\]

Again, by the results of Brown and R\"{o}dl \cite{BrownRodl1991}, and Lefmann \cite{Lefmann1991}, for all positive integers $k\geq2$, if $n$ is large enough, then every $2$-coloring of $\{1,2,\ldots,n\}$ has a monochromatic solution to the equation
\begin{equation}\label{Equation:GeneralUnitFractions}
    \frac{a_1}{x_1}+\frac{a_2}{x_2}+\cdots+\frac{a_k}{x_k}=\frac{1}{x_{k+1}},
\end{equation}
where $a_1,a_2,\ldots,a_k$ are positive integers. Let $f(a_1,a_2,\ldots,a_k)$ be the smallest positive integer $n$ such that every $2$-coloring of $\{1,2,\ldots,n\}$ has a monochromatic solution to Equation~(\ref{Equation:GeneralUnitFractions}), where $x_1,x_2,\ldots,x_k$ are not necessarily distinct. Generalizing Theorem~\ref{Theorem:Main} for 2-colorings, we prove the following:
\begin{theorem}\label{Theorem:General2-Coloring}
We have \[f(a_1,a_2)\geq2\min\{a_1,a_2\}(a_1+a_2)^2;\] and, for all $k\geq3$, we have
    \[
    f(a_1,a_2,\ldots,a_k)\geq(2\min\{a_1,a_2,\ldots,a_k\}+1)(a_1+a_2+\cdots+a_k)^2.
    \]
\end{theorem}

The rest of this paper is organized as follows. We prove Theorem~\ref{Theorem:Main} in Section~\ref{Section:ProofMainTheorem}. In Section~\ref{Section:Computation2Coloring}, we first present some computational results for $f_2(k)$ and then prove Theorem~\ref{Theorem:2-ColoringsCombined}. Theorem~\ref{Theorem:General2-Coloring} is proved in Section~\ref{Section:GeneralTheorem2Coloring}. We present some computational results for $f_3(k)$ and discuss some open problems in Section~\ref{Section:Concluding}.

\section{Proof of Theorem \ref{Theorem:Main}}\label{Section:ProofMainTheorem}
We start with an elementary observation.
\begin{lemma}\label{Lemma:Elementary}
    Let $s$ and $t$ be positive integers. If $a,b\in\{s,s+1,\ldots,s+t\}$ with $a<b$, then
    \[
    \frac{1}{a}-\frac{1}{b}\geq\frac{1}{(s+t-1)(s+t)}.
    \]
\end{lemma}
\begin{proof}\
Let $a,b\in\{s,s+1,\ldots,s+t\}$ with $a<b$. Since $a\leq b-1$ and $b\leq s+t$, we have
    \[
    \frac{1}{a}-\frac{1}{b}\geq\frac{1}{b-1}-\frac{1}{b}=\frac{1}{(b-1)b}\geq\frac{1}{(s+t-1)(s+t)}.
    \]
This completes the proof.
\end{proof}

Now we prove Theorem~\ref{Theorem:Main}. We do so by constructing $r$-colorings of the following form
\[
A_2'\ \ A_3'\ \ \cdots\ \  A_r'\ \ A_1\ \ A_2''\ \ A_3''\ \ \cdots\ \ A_r''.
\]
In the above coloring, $A_1$, $A_2'$, $A_3'$, \ldots, $A_r'$, $A_2''$, $A_3''$, \ldots, $A_r''$ are $2r-1$ mutually disjoint discrete intervals listed in increasing order, and, for any given $r$, $A_i'$ and $A_i''$ get larger in size as $i\geq2$ increases. We color each $A_i$ with the $i$-th color where $A_i=A_i'\cup A_i''$ for $i\in\{2,3,\ldots,r\}$. We have two different bounds in Theorem~\ref{Theorem:Main} because, for $r\geq2$, we have $A_2'=\{1\}$ when $k=2$ and $A_2'=\{1,2\}$ when $k\geq3$. Consequently, we have $A_1=\{2^{r-1},2^{r-1}+1,\ldots,2^r-1\}$ for $k=2$, and $A_1=\{2^r-1,2^r,\ldots,(2^r-1)k-1\}$ for $k\geq3$.

We first prove Theorem \ref{Theorem:Main} for $k=2$.

\begin{thm1.1PartI}\label{Lemma:CommonCoefficients_K=2}
    For all $r\geq1$, we have $f_r(2)\geq4^r/2$.
\end{thm1.1PartI}
\begin{proof}
Let
\[
A_1=\{2^{r-1},2^{r-1}+1,\ldots,2^r-1\}
\]
and, for all $i\in\{2,3,\ldots,r\}$, let 
\[
A_i'=\{2^{i-2},2^{i-2}+1,\ldots,2^{i-1}-1\},
\]    
\[
A''_i=\{2^{i-2}2^{r},2^{i-2}2^{r}+1,\ldots,2^{i-1}2^{r}-1\},
\]
and $A_i=A_i'\cup A''_i$. Note that $A_1,A_2',A_2'',\ldots,A_r',A_r''$ are mutually disjoint and $A_1\cup A_2\cup\cdots\cup A_r=\{1,2,\ldots,2^{r-1}\cdot2^r-1\}$. Consider the coloring $\Delta:\{1,2,\ldots,2^{r-1}\cdot2^r-1\}\to\{1,2,\ldots,r\}$ such that $\Delta(a)=i$ for all $a\in A_i$ with $i\in\{1,2,\ldots,r\}$. We will show that $\Delta$ does not have a monochromatic solution to Equation~(\ref{Equation:Main}) for $k=2$. To do this, it suffices to show that, for all $i\in\{1,2,\ldots,r\}$, the set $A_i$ does not have a solution to Equation~(\ref{Equation:Main}) for $k=2$.

We first show that $A_1$ does not have a solution to Equation~(\ref{Equation:Main}) for $k=2$. Suppose, by way of contradiction, that $a,b_1,b_2\in A_1$ such that
\[
\frac{1}{a}=\frac{1}{b_1}+\frac{1}{b_2}.
\]
Then we have
\[
\frac{1}{a}>\frac{1}{2^r}+\frac{1}{2^r}=\frac{2}{2^r}=\frac{1}{2^{r-1}}
\]
and hence $a<2^{r-1}$, which is a contradiction.

Now we show that, for any $i\in\{2,3,\ldots,r\}$, the set $A_i$ does not have a solution to Equation~(\ref{Equation:Main}) for $k=2$. Suppose, by way of contradiction, that $b_1,b_2\in A_i$ such that
\[
\frac{1}{a}=\frac{1}{b_1}+\frac{1}{b_2}.
\]

{\bf Case 1}: $b_1,b_2\in A''_i$. Then we have
\[
\frac{1}{a}\leq\frac{2}{2^{i-2}2^r}=\frac{1}{2^{i-2}2^{r-1}}
\]
and
\[
\frac{1}{a}>\frac{2}{2^{i-1}2^r}=\frac{1}{2^{i-2}2^r}.
\]
It follows that $2^{r-1}\leq 2^{i-2}2^{r-1}\leq a<2^{i-2}2^r$, which is a contradiction.

{\bf Case 2}: $b_j\in A_i'$ for some $j\in\{1,2\}$. Since $a<b_1,b_2$, we have $a\in A_i'$. We note that $b_1$ and $b_2$ cannot both be in $A_i'$. To see this, suppose $b_1,b_1\in A_i'$, then
\[
\frac{1}{a}=\frac{1}{b_1}+\frac{1}{b_2}>\frac{1}{2^{i-1}}+\frac{1}{2^{i-1}}=\frac{1}{2^{i-2}}
\]
and hence $a<2^{i-2}$, which is a contradiction. Without loss of generality, we assume that $b_2\in A_i'$ and $b_1\in A_i''$. Then we have
\[
\frac{1}{a}-\frac{1}{b_2}=\frac{1}{b_1}\leq \frac{1}{2^{i-2}2^r},
\]
and, by Lemma~\ref{Lemma:Elementary},
\[
\frac{1}{a}-\frac{1}{b_2}\geq\frac{1}{(2^{i-1}-2)(2^{i-1}-1)}.
\]
It follows that $(2^{i-1}-2)(2^{i-1}-1)\geq2^{i-2}2^r$, which is a contradiction.
\end{proof}
Now we prove Theorem~\ref{Theorem:Main} for $k\geq3$.
\begin{thm1.1PartII}\label{Lemma:CommonCoefficients_k>=3}
    For all $k\geq3$ and $r\geq1$, we have $f_r(k)\geq(2^r-1)k^r$.
\end{thm1.1PartII}

\begin{proof}
    Let 
    \[
    A_1=\{2^r-1,2^r,\ldots,(2^r-1)k-1\}
    \]
    and, for all $i\in\{2,3,\ldots,r\}$, let
    \[
    A_i'=\{2^{i-1}-1,2^{i-1},\ldots,2^i-2\},
    \]
    \[
    A_i^{''}\{(2^r-1)k^{i-1}, (2^r-1)k^{i-1}+1,\ldots,(2^r-1)k^i-1\},
    \]
    and $A_i=A_i'\cup A_i^{''}$. Note that $A_1,A_2',A_2'',\ldots,A_r',A_r''$ are mutually disjoint and $A_1\cup A_2\cup\cdots\cup A_r=\{1,2,\ldots,(2^r-1)k^r-1\}$. Let $\Delta:\{1,2,\ldots,(2^r-1)k^r-1\}\to\{1,2,\ldots,r\}$ be an $r$-coloring such that $\Delta(a)=i$ for all $a\in A_i$ with $i\in\{1,2,\ldots,r\}$. We will show that $\Delta$ does not have a monochromatic solution to Equation~(\ref{Equation:Main}). To do this, it suffices to show that, for all $i\in\{1,2,\ldots,r\}$, the set $A_i$ does not have a solution to Equation~(\ref{Equation:Main}) .

    We first show that $A_1$ does not have a solution to Equation~(\ref{Equation:Main}). Suppose, by way of contradiction, $a,b_1,b_2,\ldots,b_k\in A_1$ such that
    \[
    \frac{1}{a}=\frac{1}{b_1}+\frac{1}{b_2}+\cdots+\frac{1}{b_k}.
    \]
    Then we have
    \[
    \frac{1}{a}>\frac{k}{(2^r-1)k}=\frac{1}{2^r-1}.
    \]
    Hence $a<2^r-1$, which is a contradiction.

    Now let $i\in\{2,3,\ldots,r\}$ and we will show that $A_i$ does not have a solution to Equation~(\ref{Equation:Main}). Suppose, by way of contradiction, $a,b_1,b_2,\ldots,b_k\in A_1$ such that
    \[
    \frac{1}{a}=\frac{1}{b_1}+\frac{1}{b_2}+\cdots+\frac{1}{b_k}.
    \]
    There are two cases depending on whether $b_j\in A_i'$ or $b_j\in A_i''$ for $j\in\{1,2,\ldots,k\}$.

   {\bf Case 1}: $b_1,b_2,\ldots,b_k\in A_i''$. Then we have
    \[
    \frac{1}{a}>\frac{k}{(2^r-1)k^i}=\frac{1}{(2^r-1)k^{i-1}}
    \]
    and
    \[
    \frac{1}{a}\leq\frac{k}{(2^r-1)k^{i-1}}=\frac{1}{(2^r-1)k^{i-2}}.
    \]
    It follows that $(2^r-1)k^{i-2}\leq a<(2^r-1)k^{i-1}$ and hence $a\in A_{i-1}$, which is a contradiction.

    {\bf Case 2}: There exists $j\in\{1,2,\ldots,k\}$ such that $b_j\in A_i'$. Without loss of generality, we assume that $b_k\in A_i'$. Notice that this forces $b_1,b_2,\ldots,b_{k-1}\in A_i''$. To see this, suppose that $b_{k-1}\in A_i'$. Then we have
    \[
    \frac{1}{a}-\frac{1}{b_{k-1}}-\frac{1}{b_k}\leq\frac{1}{2^{i-1}-1}-\frac{1}{2^i-2}-\frac{1}{2^i-2}=0,
    \]
    which is a contradiction because $k\geq3$.

    Now we have
    \[
    \frac{1}{a}-\frac{1}{b_k}=\frac{1}{b_1}+\frac{1}{b_2}+\cdots+\frac{1}{b_{k-1}}\leq\frac{k-1}{(2^r-1)k^{i-1}}=\frac{k-1}{k}\frac{1}{(2^r-1)k^{i-2}}
    \]
    and, by Lemma~\ref{Lemma:Elementary}, 
    \[
    \frac{1}{a}-\frac{1}{b_k}\geq\frac{1}{(2^i-3)(2^i-2)}.
    \]

    It follows that
    \begin{equation}\label{Inequality:1}
(2^i-3)(2^i-2)\geq\frac{k}{k-1}(2^r-1)k^{i-2}.  
\end{equation}
We still need to show that, for all $i\geq2$ and $k\geq3$, Inequality~(\ref{Inequality:1}) leads to a contradiction. Notice that, for any $i\geq2$, we have $r\geq i\geq2$. 

If $i=2$, then Inequality~(\ref{Inequality:1}) leads to $2\geq3$, which is a contradiction. 

If $i=3$, then Inequality~(\ref{Inequality:1}) becomes $30\geq7k^2/(k-1)$. This is a contradiction for $k\geq5$ because it leads to $30\geq35$. For $k=3$, we have $30\geq63/2$, which is a contradiction; and for $k=4$, we have $30\geq112/3$, which is again a contradiction.

If $i=4$, then Inequality~(\ref{Inequality:1}) becomes $13\cdot14\geq15k^3/(k-1)$. This is a contradiction for $k\geq4$ because it leads to $13\cdot14\geq15\cdot16$. For $k=3$, we have $13\cdot14\geq15\cdot27/2$, which is a contradiction.

Now suppose $i\geq5$. For $k\geq4$, we have $k^{i-2}\geq4^{i-2}\geq 2^i>2^i-3$ and hence Inequality~(\ref{Inequality:1}) implies
$
(2^i-3)(2^i-2)>(2^i-1)(2^i-3),
$
which is a contradiction; and for $k=3$, by Inequality~(\ref{Inequality:1}), we have
    \[
    (2^i-3)(2^i-2)\geq\frac{3}{2}(2^i-1)3^{i-2}=(2^i-1)2^i\frac{1}{4}\left(\frac{3}{2}\right)^{i-1}\geq(2^i-1)2^i\frac{1}{4}\left(\frac{3}{2}\right)^{5-1}>(2^i-1)2^i
    \]
    and hence $(2^i-3)(2^i-2)>(2^i-1)2^i$, which is again a contradiction.
\end{proof}

\section{$2$-Colorings}\label{Section:Computation2Coloring}
In this section, we first present some computational results for $f_2(k)$. Our computation is conducted using Python, and the code is developed with the assistance of ChatGPT 5.5. The Python script is available from the first author's website. We use Boolean satisfiability (SAT) solvers \texttt{Cadical195} and \texttt{Glucose3} in the \texttt{PySAT} toolkit for the computation \cite{ITK2024}. The CPU of the computer we use is AMD Ryzen 7 9800X3D. For $k\geq14$, we also use ChatGPT Codex to help with the calculation.

\begin{table}[H]
\renewcommand{\arraystretch}{1.5}
\centering
\begin{tabular}{|c|c|c|c|c|c|c|c|c|c|c|c|c|}
\hline
$k$&2&4&6&8&10&12&14&16&18&20&22&24\\\hline
$f_2(k)$&60&48&108&192&300&432&588&768&972&1200&1452&1728\\\hline
$f_2(k)-3k^2$&48&0&0&0&0&0&0&0&0&0&0&0\\\hline
\end{tabular}
\caption{Computational Results for $f_2(k)$, $2\leq k\leq 24$ even}\label{Table1}
\end{table}

\begin{table}[H]
\renewcommand{\arraystretch}{1.5}
\centering
\begin{tabular}{|c|c|c|c|c|c|c|c|c|c|c|c|c|}
\hline
$k$&3&5&7&9&11&13&15&17&19&21&23&25\\\hline
$f_2(k)$&40&80&150&252&372&513&675&880&1092&1323&1608&1887\\\hline
$f_2(k)-3k^2$&13&5&3&9&9&6&0&13&9&0&21&12\\\hline
\end{tabular}
\caption{Computational Results for $f_2(k)$, $3\leq k\leq 25$ odd}\label{Table2}
\end{table}

By Theorem \ref{Theorem:Main}, we have $f_2(k)\geq3k^2$ for all $k\geq3$. Tables \ref{Table1} and \ref{Table2} contain our computational results for $k\leq 25$ along with the difference between $f_2(k)$ and $3k^2$. Myers and Parrish \cite{MyersParrish2018} calculated $f_2(2)$, $f_2(3)$, and $f_2(4)$, but all other values are computed by us for the first time. We note that Myers and Parrish \cite[Computation 21 on Page 12]{MyersParrish2018} also claimed that $f_2(5)=39$ which is impossible because, by Theorem~\ref{Theorem:Main}, we have $f_2(5)\geq3\cdot5^2=75$. 

Tables \ref{Table1} and \ref{Table2} suggest that if $k$ is an odd prime power, then $f_2(k)\geq3k^2+1$; otherwise, $f_2(k)=3k^2$ for $k\geq4$. We first prove the former and then prove a special case for the latter.

\begin{theorem}\label{Theorem:2-ColorPrimePowers}
    If $k=p^m$ for some odd prime number $p$ and positive integer $m$, then $f_2(k)\geq3k^2+1$.
\end{theorem}
\begin{proof}
Let $k=p^m$ where $p$ is an odd prime number and $m$ is a positive integer. Consider the coloring $\Delta:\{1,2,\ldots,3k^2\}\to\{R,B\}$ such that $\Delta(a)=R$ for all $a\in \{1,2\}\cup\{3k,3k+1,\ldots,3k^2-1\}$ and $\Delta(b)=B$ for all $b\in \{3,4,\ldots,3k-1\}\cup\{3k^2\}$. If suffices to show that $\Delta$ does not have a monochromatic solution to Equation~(\ref{Equation:Main}). By the proof of Theorem~\ref{Theorem:Main}, $\Delta$ does not have a solution to Equation~(\ref{Equation:Main}) with the color $R$. It remains to show that $\Delta$ does not have a solution to Equation~(\ref{Equation:Main}) with color $B$. Suppose, by way of contradiction, that $a,b_1,b_2,\ldots,b_k\in\{3,4,\ldots,3k-1\}\cup\{3k^2\}$ with $b_1\geq b_2\geq\cdots\geq b_k>a$ such that
\[
\frac{1}{a}=\frac{1}{b_1}+\frac{1}{b_2}+\cdots+\frac{1}{b_k}.
\]

By the proof of Theorem~\ref{Theorem:Main}, we must have $b_1=3k^2$. Let $\ell$ be the largest integer in $\{1,2,\ldots,k\}$ such that $b_1=b_2=\cdots=b_\ell=3k^2$. We first notice that $\ell\neq k$ because otherwise
\[
\frac{1}{a}=\frac{k}{3k^2}=\frac{1}{3k}
\]
and hence $a=3k$, which is a contradiction.

Since $a,b_{\ell+1},b_{\ell+2},\ldots,b_{k}<3k$ and $p^{m+1}\geq 3p^m=3k$, we see that $p^{m+1}$ does not divide any integer in $\{a,b_{\ell+1},b_{\ell+2},\ldots,b_{k}\}$. Let $L$ be the least common multiple of the set $\{3k^2,b_{\ell+1},b_{\ell+2},\ldots,b_k\}$. Notice that $p^{2m}=k^2$ divides $L$ and $p$ does not divide $L/(3k^2)$. Now we have
\[
\frac{1}{a}=\underbrace{\frac{1}{3k^2}+\frac{1}{3k^2}+\cdots+\frac{1}{3k^2}}_{\ell\text{ times}}+\frac{1}{b_{\ell+1}}+\frac{1}{b_{\ell+2}}+\cdots+\frac{1}{b_k}
\]
and hence
\[
\frac{1}{a}=\frac{\ell L/(3k^2)+L/b_{\ell+1}+L/b_{\ell+2}+\cdots+L/b_k}{L}.
\]
Since $p^{m+1}$ does not divide any integer in $\{b_{\ell+1},b_{\ell+2},\ldots,b_{k}\}$ and $p^{2m}$ divides $L$, we see that $p^m$ divides $L/b_{\ell+1}+L/b_{\ell+2}+\cdots+L/b_k$. Since $p^{m+1}$ does not divide $a$, we must have that $p^m$ divides $\ell L/(3k^2)$. Since $p$ does not divide $L/(3k^2)$, we see that $k=p^m$ divides $\ell$, which is a contradiction.
\end{proof}
\begin{remark}
In the proof of Theorem~\ref{Theorem:2-ColorPrimePowers}, we showed that if $k$ is an odd prime power, then the 2-coloring of $\{1,2,\ldots,3k^2\}$, where the integers in $\{1,2\}\cup\{3k,3k+1,\ldots,3k^2-1\}$ are colored with one color and the integers in $\{3,4,\ldots,3k-1\}\cup\{3k^2\}$ are colored with the other color, does not have a monochromatic solution to Equation~(\ref{Equation:Main}). For certain values of $k$, it is possible to add more values greater than $3k^2-1$ to the second color class and the 2-coloring still has no monochromatic solution to Equation~(\ref{Equation:Main}). For example, there is a 2-coloring of $\{1,2,\ldots,34\}$ that does not have a monochromatic solution to Equation~(\ref{Equation:Main}) for $k=3$: we can color the integers in $\{1,2\}\cup\{9,10,\ldots,26\}$ red and color the integers in $\{3,4,5,6,7,8\}\cup\{27,28,\ldots,34\}$ blue. Another example is a 2-coloring of $\{1,2,\ldots,76\}$ that does not have a monochromatic solution to Equation~(\ref{Equation:Main}) for $k=5$: we can color the integers in $\{1,2\}\cup\{15,16,\ldots,74\}$ red and color the integers in $\{3,4,\ldots,14\}\cup\{75,76\}$ blue.
\end{remark}

Now we prove that $f_2(k)=3k^2$ for an infinite family of $k$.
\begin{theorem}
    If $k=3\cdot2^m$ for some positive integer $m$, then $f_2(k)=3k^2$.
\end{theorem}
\begin{proof}
    Let $k=3\cdot2^m$ for some positive integer $m$. By Theorem~\ref{Theorem:Main}, we have $f_2(k)\geq3k^2$. It remains to show that $f_2(k)\leq 3k^2$. By Table~\ref{Table1}, this is true for $k=6$. Hence, we assume that $k\geq12$. Notice that we then have $m\geq2$.

Let $\Delta:\{1,2,\ldots,3k^2\}\to\{R,B\}$ be a $2$-coloring. Suppose, by way of contradiction, that $\Delta$ does not have a monochromatic solution to Equation~(\ref{Equation:Main}). First, notice that, for any integer $a$, if $a,ak\in\{1,2,\ldots,3k^2\}$, then $a$ and $ak$ have different colors because
\[
\frac{1}{a}=\underbrace{\frac{1}{ak}+\frac{1}{ak}+\cdots+\frac{1}{ak}}_{k\text{ times}}.
\]
We divide the rest of the proof into four cases. Note that every 2-colorings of $\{1,2,\ldots,3k^2\}$ belongs to one of those cases.

\textbf{Case 1}: $\Delta(2)\neq\Delta(3)$. Without loss of generality, we assume that $\Delta(2)=R$ and $\Delta(3)=B$. Then we have
    \[
    \Delta(2)=\Delta(3k)=R
    \]
    and
    \[
    \Delta(3)=\Delta(2k)=\Delta(3k^2)=B.
    \]
Since
    \[
    \frac{1}{2}=\underbrace{\frac{1}{3k/2}+\frac{1}{3k/2}+\cdots+\frac{1}{3k/2}}_{k/2\text{ times}}+\underbrace{\frac{1}{3k}+\frac{1}{3k}+\cdots+\frac{1}{3k}}_{k/2\text{ times}},
    \]
    we have $\Delta(3k/2)=B$. Now since
\[\frac{1}{3}=\underbrace{\frac{1}{3k/2}+\frac{1}{3k/2}+\cdots+\frac{1}{3k/2}}_{k/2-1\text{ times}}+\frac{1}{2k}+\underbrace{\frac{1}{3k^2}+\frac{1}{3k^2}+\cdots+\frac{1}{3k^2}}_{k/2\text{ times}},\]
we have a monochromatic solution, which is a contradiction.

\textbf{Case 2:} $\Delta(1)=\Delta(2)=\Delta(3)\neq\Delta(4)$. Without loss of generality, assume that $\Delta(1)=\Delta(2)=\Delta(3)=R$ and $\Delta(4)=B$. Since

\[
\frac{1}{1}=\frac{1}{3}+\frac{1}{3}+\underbrace{\frac{1}{3(k-2)}+\frac{1}{3(k-2)}+\cdots+\frac{1}{3(k-2)}}_{k-2\text{ times}},
\]

and
\[
\frac{1}{1}=\frac{1}{2}+\frac{1}{3}+\underbrace{\frac{1}{6(k-2)}+\frac{1}{6(k-2)}+\cdots+\frac{1}{6(k-2)}}_{k-2\text{ times}},
\]
we have $\Delta(3(k-2))=\Delta(6(k-2))=B$. Since
\[
\frac{1}{4}=\underbrace{\frac{1}{3(k-2)}+\frac{1}{3(k-2)}+\cdots+\frac{1}{3(k-2)}}_{k/2-3\text{ times}}+\underbrace{\frac{1}{6(k-2)}+\frac{1}{6(k-2)}+\cdots+\frac{1}{6(k-2)}}_{k/2+3\text{ times}},
\]
we have a contradiction.

\textbf{Case 3}: $\Delta(1)=\Delta(2)=\Delta(3)=\Delta(4)$. Without loss of generality, assume that $\Delta(1)=\Delta(2)=\Delta(3)=\Delta(4)=R$. Let $A_0=\{1,2\}$ and, for all positive integers $i$, let $A_i=\{3\cdot2^{i-1},2^{i+1}\}$. Notice that we have $A_0=\{1,2\}$ and $A_1=\{3,4\}$, and $\Delta(a)=R$ for all $a\in A_0\cup A_1$. Since $k=3\cdot2^m$, we have $k\in A_{m+1}$. We prove, by induction on $i$, that $\Delta(a)=R$ for all $a\in A_{m+1}$. Suppose that, for some $j\in\{1,2,\ldots,m\}$, we have $\Delta(a)=R$ for all $a\in A_i$ with $i\leq j$. Then we have $\Delta(2^j)=\Delta(3\cdot2^{j-1})=\Delta(2^{j+1})=R$. We need to prove that $\Delta(a)=R$ for all $a\in A_{j+1}$; that is, we need to prove that $\Delta(3\cdot2^j)=\Delta(2^{j+2})=R$.

Since
\begin{equation}\label{Identity:1}
    \frac{1}{2^j}=\frac{1}{3\cdot2^{j-1}}+\underbrace{\frac{1}{3\cdot2^{j}(k-1)}+\frac{1}{3\cdot2^{j}(k-1)}+\cdots+\frac{1}{3\cdot2^{j}(k-1)}}_{k-1\text{ times}},
\end{equation}
and
\begin{equation}\label{Identity:2}
    \frac{1}{3\cdot2^{j-1}}=\frac{1}{2^{j+1}}+\underbrace{\frac{1}{3\cdot2^{j+1}(k-1)}+\frac{1}{3\cdot2^{j+1}(k-1)}+\ldots+\frac{1}{3\cdot2^{j+1}(k-1)}}_{k-1\text{ times}}
\end{equation}
we have $\Delta(3\cdot2^{j}(k-1))=\Delta(3\cdot2^{j+1}(k-1))=B$. Since
\begin{equation}\label{Identity:3}
    \frac{1}{3\cdot2^j}=\underbrace{\frac{1}{3\cdot2^j(k-1)}+\frac{1}{3\cdot2^j(k-1)}+\cdots+\frac{1}{3\cdot2^j(k-1)}}_{k-2\text{ times}}+\frac{1}{3\cdot2^{j+1}(k-1)}+\frac{1}{3\cdot2^{j+1}(k-1)},
\end{equation}
we have $\Delta(3\cdot2^j)=R$. 

Since
\begin{equation}\label{Identity:4}
\frac{1}{2^j}=\frac{1}{3\cdot2^{j}}+\frac{1}{3\cdot2^{j}}+\underbrace{\frac{1}{3\cdot2^{j}(k-2)}+\frac{1}{3\cdot2^{j}(k-2)}+\cdots+\frac{1}{3\cdot2^{j}(k-2)}}_{k-2\text{ times}},
\end{equation}
\begin{equation}\label{Identity:5}
    \frac{1}{2^{j}}=\frac{1}{2^{j+1}}+\frac{1}{3\cdot2^{j}}+\underbrace{\frac{1}{3\cdot2^{j+1}(k-2)}+\frac{1}{3\cdot2^{j+1}(k-2)}+\cdots+\frac{1}{3\cdot2^{j+1}(k-2)}}_{k-2\text{ times}},
\end{equation}
we have $\Delta(3\cdot2^{j}(k-2))=\Delta(3\cdot2^{j+1}(k-2))=B$. Since
\begin{equation}\label{Identity:6}
    \begin{split}
    \frac{1}{2^{j+2}}=&\underbrace{\frac{1}{3\cdot2^{j}(k-2)}+\frac{1}{3\cdot2^{j}(k-2)}+\cdots+\frac{1}{3\cdot2^{j}(k-2)}}_{k/2-3\text{ times}}\\&+\underbrace{\frac{1}{3\cdot2^{j+1}(k-2)}+\frac{1}{3\cdot2^{j+1}(k-2)}+\cdots+\frac{1}{3\cdot2^{j+1}(k-2)}}_{k/2+3\text{ times}},
\end{split}
\end{equation}
we have $\Delta(2^{j+2})=R$.

The largest integer used in Identities~(\ref{Identity:1}-\ref{Identity:6}) is
\[
3\cdot2^{j+1}(k-1)\leq 3\cdot2^{m+1}(k-1)=2k(k-1)\leq 3k^2
\]
and hence all the integers we used are in $\{1,2,\ldots,3k^2\}$. By induction, we have $\Delta(a)=R$ for $a\in A_{m+1}$. Since $k\in A_{m+1}$, we have $\Delta(k)=R$. This is a contradiction because
\[
\frac{1}{1}=\underbrace{\frac{1}{k}+\frac{1}{k}+\cdots+\frac{1}{k}}_{k\text{ times}}.
\]

\textbf{Case 4}: $\Delta(1)\neq\Delta(2)=\Delta(3)$. Without loss of generality, we assume that $\Delta(1)=R$ and $\Delta(2)=\Delta(3)=B$. Then we have
\[
\Delta(1)=\Delta(2k)=\Delta(3k)=R
\]
and
\[
\Delta(2)=\Delta(3)=\Delta(k)=\Delta(2k^2)=\Delta(3k^2)=B.
\]
Since
\[
\frac{1}{2}=\underbrace{\frac{1}{k}+\frac{1}{k}+\cdots+\frac{1}{k}}_{k/3\text{ times}}+\underbrace{\frac{1}{4k}+\frac{1}{4k}+\cdots+\frac{1}{4k}}_{2k/3\text{ times}},
\]
we have $\Delta(4k)=R$ and hence $\Delta(4)=B$.

Let $C_0=\{2\}$ and, for all positive integers $i$, let $C_i=\{3\cdot2^{i-1},2^{i+1}\}$. Notice that $C_1=\{3,4\}$. We prove, by induction on $i$, that $\Delta(c)=B$ for all $c\in C_i$ with $i\in\{0,1,2,\ldots,m+1\}$. This is true for $i\in\{0,1\}$. Suppose that, for some $j\in\{1,2,\ldots,m\}$, we have $\Delta(c)=B$ for all $c\in C_i$ with $i\leq j$. Then we have $\Delta(2^j)=\Delta(3\cdot2^{j-1})=\Delta(2^{j+1})=B$. We need to prove that $\Delta(c)=B$ for all $c\in C_{j+1}$; that is, we need to prove that $\Delta(3\cdot2^{j})=\Delta(2^{j+2})=B$. 

By Identities~(\ref{Identity:1}) and (\ref{Identity:2}), we have $\Delta(3\cdot2^{j}(k-1))=\Delta(3\cdot2^{j+1}(k-1))= R$, and hence, by Identity~(\ref{Identity:3}), we have $\Delta(3\cdot2^{j})=B$. Similarly, by Identities~(\ref{Identity:4}) and (\ref{Identity:5}), we have $\Delta(3\cdot2^{j}(k-2))=\Delta(3\cdot2^{j+1}(k-2))=R$, and hence, by Identity~(\ref{Identity:6}), we have we have $\Delta(2^{j+2})=B$. Similar to Case 3, the largest integer used in the induction step is
\[
3\cdot2^{j+1}(k-1)\leq 3\cdot2^{m+1}(k-1)=2k(k-1)\leq 3k^2.
\]
By induction, we have $\Delta(c)=B$ for all $c\in C_i$ with $i\in\{0,1,2,\ldots,m+1\}$. Since $k=3\cdot2^m$ with $m\geq2$, we have $k/6=2^{m-1}\in C_{m-2}$, $k/3=2^m\in C_{m-1}$, $k/2=3\cdot 2^{m-1}\in C_{m}$, $2k/3=2^{m+1}\in C_m$, and $4k/3=2^{m+2}\in C_{m+1}$. Hence $\Delta(k/6)=\Delta(k/3)=\Delta(k/2)=\Delta(2k/3)=\Delta(4k/3)=B$. Since
\[
\frac{1}{k/6}=\frac{1}{k/2}+\frac{1}{2k/3}+\frac{1}{2k/3}+\underbrace{\frac{1}{k(k-3)}+\frac{1}{k(k-3)}+\cdots+\frac{1}{k(k-3)}}_{k-3\text{ times}},
\]
and
\[
\frac{1}{k/3}=\frac{1}{k}+\frac{1}{4k/3}+\frac{1}{4k/3}+\underbrace{\frac{1}{2k(k-3)}+\frac{1}{2k(k-3)}+\cdots+\frac{1}{2k(k-3)}}_{k-3\text{ times}},
\]
we have $\Delta(k(k-3))=\Delta(2k(k-3))=R$. Since
\[
\frac{1}{3k/2}=\underbrace{\frac{1}{k(k-3)}+\frac{1}{k(k-3)}+\cdots+\frac{1}{k(k-3)}}_{k/3-4\text{ times}}+\underbrace{\frac{1}{2k(k-3)}+\frac{1}{2k(k-3)}+\cdots+\frac{1}{2k(k-3)}}_{2k/3+4\text{ times}},
\]
we have $\Delta(3k/2)=B$. But then since
    \[
    \frac{1}{k}=\frac{1}{3k/2}+\frac{1}{2k^2}+\frac{1}{2k^2}+\underbrace{\frac{1}{3k^2}+\frac{1}{3k^2}+\cdots+\frac{1}{3k^2}}_{k-3\text{ times}},
    \]
    we have a monochromatic solution, which is a contradiction.
\end{proof}

\begin{remark}
    The first author \cite{Gaiser2024} proved that $f_2(k)\leq 6k(k+1)(k+2)$ for all $k\geq2$ using the following fact: if every $2$-coloring of a finite set $A$ has a monochromatic solution to Equation~(\ref{Equation:Linear}), then every $2$-coloring of $\{1,2,\ldots,L\}$ has a monochromatic solution to Equation~(\ref{Equation:Main}), where $L$ is the least common multiple of the set $A$. This fact is a variant of a result of Brown and R\"{o}dl \cite{BrownRodl1991}. However, this fact cannot be used directly to prove that $f_2(k)=3k^2$ for certain values of $k$. For example, we know that $f_2(4)=48$ and the set $\{1,2,3,4,6,8,12,16,24,48\}$ contains all the divisors of $48$, but there is a $2$-coloring of $\{1,2,3,4,6,8,12,16,24,48\}$ which does not have a monochromatic solution to Equation~(\ref{Equation:Main}) for $k=4$: we can color the integers in $\{1,8,12,16,24\}$ red and color the integers in $\{2,3,4,6,48\}$ blue.
\end{remark}

\section{Proof of Theorem \ref{Theorem:General2-Coloring}}\label{Section:GeneralTheorem2Coloring}
We first prove Theorem \ref{Theorem:General2-Coloring} for $k=2$.
\begin{thm1.4PartI}\label{Theorem:General2-Coloring_k=2}
    We have $f(a_1,a_2)\geq2\min\{a_1,a_2\}(a_1+a_2)^2$.
\end{thm1.4PartI}
\begin{proof}
    Write $\alpha=\min\{a_1,a_2\}$ and $\beta=a_1+a_2$. We will construct a $2$-coloring of $\{1,2,\ldots,2\alpha\beta^2-1\}$ that does not have a monochromatic solution to Equation~(\ref{Equation:GeneralUnitFractions}) for $k=2$. Let \[A'=\{1,2,\ldots,2\alpha-1\},\]
    \[
    A''=\{2\alpha\beta,2\alpha\beta+1,\ldots,2\alpha\beta^2-1\},
    \]
    $A=A'\cup A''$, and
    \[
    B=\{2\alpha,2\alpha+1,\cdots,2\alpha\beta-1\}.
    \]
    
    Consider the $2$-coloring $\Delta:\{1,2,\ldots,2\alpha\beta^2-1\}\to\{1,2\}$ with $\Delta(a)=1$ for all $a\in A$ and $\Delta(b)=2$ for all $b\in B$. It suffices to show that neither $A$ nor $B$ has a solution to Equation~(\ref{Equation:GeneralUnitFractions}) for $k=2$. We first show that $B$ does not have a solution to Equation~(\ref{Equation:GeneralUnitFractions}) for $k=2$. Suppose, for a contradiction, that $c,b_1,b_2\in B$ such that
    \[
    \frac{1}{c}=\frac{a_1}{b_1}+\frac{a_2}{b_2}.
    \]
    Then we have
    \[
    \frac{1}{c}>\frac{a_1+a_2}{2\alpha\beta}=\frac{\beta}{2\alpha\beta}=\frac{1}{2\alpha}
    \]
    and hence $c<2\alpha$, which is a contradiction.

    Next, we show that $A$ does not have a solution to Equation~(\ref{Equation:GeneralUnitFractions}) for $k=2$. Suppose, for a contradiction, that $c,b_1,b_2\in A$ such that
    \[
    \frac{1}{c}=\frac{a_1}{b_1}+\frac{a_2}{b_2}.
    \]

 {\bf Case 1:} $b_1,b_2\in A''$. Then we have
    \[
    \frac{1}{c}>\frac{a_1+a_2}{2\alpha\beta^2}=\frac{\beta}{2\alpha\beta^2}=\frac{1}{2\alpha\beta}
    \]
    and
    \[
    \frac{1}{c}\leq\frac{a_1+a_2}{2\alpha\beta}=\frac{1}{2\alpha}.
    \]
    It follows that $2\alpha\leq c<2\alpha\beta$, which is a contradiction.

{\bf Case 2}: $b_1\in A'$ or $b_2\in A'$. If we have both $b_1\in A'$ and $b_2\in A'$, then since $a_1/b_1\geq\alpha/(2\alpha-1)>1/2$ and $a_2/b_2\geq\alpha/(2\alpha-1)>1/2$, we have $1/c>1/2+1/2=1$, which is impossible. Hence we cannot have both $b_1\in A'$ and $b_2\in A'$. Without loss of generality, we assume that $b_2\in A'$ and $b_1\in A''$. Since $a_2/b_2<1/c\leq 1$, we have $a_2/b_2\leq (2\alpha-2)/(2\alpha-1)$. Since $1/c-a_2/b_2>0$ and $a_2/b_2>1/2$, we have $c=1$. It follows that
    \[
    \frac{1}{c}-\frac{a_2}{b_2}=1-\frac{a_2}{b_2}\geq1-\frac{2\alpha-2}{2\alpha-1}=\frac{1}{2\alpha-1}.
    \]
    On the other hand, 
    \[
    \frac{1}{c}-\frac{a_2}{b_2}=\frac{a_1}{b_1}\leq\frac{a_1}{2\alpha\beta}<\frac{1}{2\alpha}.
    \]
    Hence we have
    \[
    \frac{1}{2\alpha}>\frac{1}{2\alpha-1},
    \]
    which is a contradiction.
\end{proof}
Now we prove Theorem~\ref{Theorem:General2-Coloring} for $k\geq3$.
\begin{thm1.4PartII}\label{Theorem:General2-Coloring_k=3}
For all $k\geq3$, we have
    \[
    f(a_1,a_2,\ldots,a_k)\geq(2\min\{a_1,a_2,\ldots,a_k\}+1)(a_1+a_2+\cdots+a_k)^2.
    \]    
\end{thm1.4PartII}

\begin{proof}
    Write $\alpha=\min\{a_1,a_2,\ldots,a_k\}$ and $\beta=a_1+a_2+\cdots+a_k$. We will build a $2$-coloring of $\{1,2,\ldots,(2\alpha+1)\beta^2-1\}$ that does not have a monochromatic solution to Equation~(\ref{Equation:GeneralUnitFractions}). Let \[A'=\{1,2,\ldots,2\alpha\},\]
    \[
    A''=\{(2\alpha+1)\beta,(2\alpha+2)\beta+1,\ldots,(2\alpha+1)\beta^2-1\},
    \]
    $A=A'\cup A''$, and
    \[
    B=\{2\alpha+1,2\alpha+2,\ldots,(2\alpha+1)\beta-1\}.
    \]
    
    Let $\Delta:\{1,2,\ldots,(2\alpha+1)\beta^2-1\}\to\{1,2\}$ be a $2$-coloring such that $\Delta(a)=1$ for all $a\in A$ and $\Delta(b)=2$ for all $b\in B$. It suffices to show that neither $A$ nor $B$ has a solution to Equation~(\ref{Equation:GeneralUnitFractions}). We first show that $B$ does not have a solution to Equation~(\ref{Equation:GeneralUnitFractions}). Suppose, by way of contradiction, that $c,b_1,b_2,\ldots,b_k\in B$ such that
    \[
    \frac{1}{c}=\frac{1}{b_1}+\frac{1}{b_2}+\cdots+\frac{1}{b_k}.
    \]
    Then
    \[
    \frac{1}{c}>\frac{a_1+a_2+\cdots+a_k}{(2\alpha+1)\beta}=\frac{\beta}{(2\alpha+1)\beta}=\frac{1}{2\alpha+1}.
    \]
    It follows that $c<2\alpha+1$, which is a contradiction. 

    Now we show that $A$ does not have a solution to Equation~(\ref{Equation:GeneralUnitFractions}). Suppose, by way of contradiction, that $c,b_1,b_2,\ldots,b_k\in A$ such that
    \[
    \frac{1}{c}=\frac{1}{b_1}+\frac{1}{b_2}+\cdots+\frac{1}{b_k}.
    \]

{\bf Case 1}: $b_1,b_2,\ldots,b_k\in A''$. Then we have
    \[
    \frac{1}{c}>\frac{a_1+a_2+\cdots+a_k}{(2\alpha+1)\beta^2}=\frac{1}{(2\alpha+1)\beta}
    \]
    and
    \[
    \frac{1}{c}\leq\frac{a_1+a_2+\cdots+a_k}{(2\alpha+2)\beta}=\frac{1}{2\alpha+1}.
    \]
    It follows that $2\alpha+1\leq c<(2\alpha+1)\beta$, which is a contradiction.

{\bf Case 2}: $b_j\in A'$ for some $j\in\{1,2,\ldots,k\}$. Without loss of generality, we assume that $b_k\in A'$. Notice that this implies that $b_1,b_2,\ldots,b_{k-1}\in A''$. To see this, suppose $b_j\in A'$ for some $j\in\{1,2,\ldots,k-1\}$. Since $a_k,a_j\geq\alpha$ and $b_k,b_j\leq 2\alpha$, we have $a_k/b_k,a_j/b_j\geq1/2$. Then we have
    \[
    \frac{1}{c}-\frac{a_k}{b_k}-\frac{a_j}{b_j}\leq\frac{1}{c}-1\leq0.
    \]
    This is a contradiction because $k\geq3$. Now we have
    \[
    \frac{1}{c}-\frac{a_k}{b_k}=\frac{a_1}{b_1}+\frac{a_2}{b_2}+\cdots+\frac{a_{k-1}}{b_{k-1}}\leq\frac{a_1+a_2+\cdots+a_{k-1}}{(2\alpha+1)\beta}=\frac{\beta-a_k}{(2\alpha+1)\beta}.
    \]
    Since $a_k/b_k\geq1/2$, we must have $c=1$. At the same time, we have $a_k/b_k\leq (2\alpha-1)/(2\alpha)$. It follows that $1-a_k/b_k\geq1/(2\alpha)$. Hence we have
    \[
    \frac{1}{2\alpha}\leq\frac{\beta-a_k}{(2\alpha+1)\beta}<\frac{1}{2\alpha+1},
    \]
    which is a contradiction.
\end{proof}

\section{Concluding Remarks}\label{Section:Concluding}
We presented computational results for $f_2(k)$ in Section~\ref{Section:Computation2Coloring}. Here we present some computational results for 3-colorings. Table~\ref{Table3} contains these computational results and the difference between $f_3(k)$ and $7k^3$. Myers and Parrish \cite{MyersParrish2018} computed $f_3(2)$ and we compute $f_3(k)$ for $k\in\{3,4,5\}$. 
\begin{table}[H]
\renewcommand{\arraystretch}{1.5}
\centering
\begin{tabular}{|c|c|c|c|c|}
\hline
$k$&2&3&4&5\\\hline
$f_3(k)$&3276&585&960&$1650$\\\hline
$f_3(k)-7k^3$&3220&396&512&$775$\\\hline
\end{tabular}
\caption{Computational Results for $f_3(k)$, $2\leq k\leq 5$}\label{Table3}
\end{table}
As suggested by Table~\ref{Table3}, our lower bound $f_3(k)\geq7k^3$ for $k\geq3$ in Theorem~\ref{Theorem:Main} is unlikely to be sharp.

\begin{problem}
    Improve, if possible, the lower bound for $f_3(k)$.
\end{problem}

By Theorem~\ref{Theorem:2-ColorPrimePowers}, we have $f_2(k)\geq3k^2+1$ if $k$ is an odd prime power. However, Table~\ref{Table2} suggests that $f_2(k)>3k^2+1$ whenever $k$ is an odd prime power.

\begin{problem}
    Improve, if possible, the lower bound for $f_2(k)$ when $k$ is an odd prime power.
\end{problem}

The first author \cite{Gaiser2024} proved a cubic upper bound for $f_2(k)$ for all $k\geq2$. However, Tables~\ref{Table1} and \ref{Table2} suggests a quadratic upper bound for $f_2(k)$ when $k\geq4$.

\begin{problem}
    Find, if possible, a quadratic upper bound for $f_2(k)$ for all $k\geq4$.
\end{problem}

Finally, we used Lemma~\ref{Lemma:Elementary} in the proof of Theorem~\ref{Theorem:Main}. A special case of Lemma~\ref{Lemma:Elementary} is that the smallest distance between two distinct elements in $\{1/1,1/2,\ldots,1/n\}$ is $1/[(n-1)n]$. This leads to the following problem:

\begin{problem}\label{Problem:UnitFractionDistance}
    Let $U(n)=\{1/1,1/2,\ldots,1/n\}$ and $k\geq2$. Determine
    \[
    \min\{a-b_1-b_2-\cdots-b_k>0\ |\ a,b_1,b_2,\ldots,b_k\in U(n)\}. 
    \]
\end{problem}

\section*{Acknowledgements}
The authors would like to thank Paul Horn and Aaron Robertson for reading a previous draft of this paper and providing helpful feedback. This paper is based upon the work conducted as an undergraduate summer research project under the mentorship of the first author. We thank Community College of Aurora for providing the facilities and resources that supported this research. The authors used Python for programming with the help of ChatGPT 5.5 and ChatGPT Codex.

\end{document}